\documentclass[12pt,reqno]{amsart}
\usepackage{etoolbox}

   \makeatletter

 \patchcmd{\@setaddresses}{\scshape\ignorespaces}{\ignorespaces}{}{} 

\appto\maketitle{%
\let\@makefnmark\relax  \let\@thefnmark\relax
\ifx\@empty\addresses\else\@footnotetext{%
  \vskip-\bigskipamount\@setaddresses}
  }
\def\enddoc@text{}
\makeatother

\makeatletter
\patchcmd\maketitle
  {\uppercasenonmath\shorttitle}
  {}
  {}{}
\patchcmd\maketitle
  {\@nx\MakeUppercase{\the\toks@}}
  {\the\toks@}
  {}
  {}{}
\patchcmd\@settitle{\uppercasenonmath\@title}{\Large}{}{}
\patchcmd\@setauthors
  {\MakeUppercase{\authors}}
  {\authors}
  {}{}
\makeatother
\usepackage{amsmath,amssymb,amsthm}
\usepackage{color}
\usepackage{url}
\usepackage{tikz-cd}
\usepackage[utf8]{inputenc}
\usepackage[T1]{fontenc}
\newtheorem{theorem}{Theorem}[section]

\newtheorem{corollary}{Corollary}[section]

\newtheorem{lemma}{Lemma}[section]

\numberwithin{equation}{section}
\usepackage[colorlinks=true]{hyperref}
\hypersetup{urlcolor=blue, citecolor=red , linkcolor= blue}
\usepackage[capitalise,noabbrev,nameinlink]{cleveref}      
\begin{document}
\address{$^{[1]}$ University of Sfax, Sfax, Tunisia.}
\email{\url{kais.feki@hotmail.com}}
\subjclass[2010]{Primary 47A12, 46C05; Secondary 47B65, 47A05.}

\keywords{Positive operator, $A$-adjoint operator, numerical radius, operator matrix, inequality.}

\date{\today}
\author[Kais Feki] {\Large{Kais Feki}$^{1}$}
\title[Some bounds for the $\mathbb{A}$-numerical radius of certain $2 \times 2$ operator matrices]{Some bounds for the $\mathbb{A}$-numerical radius of certain $2 \times 2$ operator matrices}
\maketitle
\begin{abstract}
For a given bounded positive (semidefinite) linear operator $A$ on a complex Hilbert space $\big(\mathcal{H}, \langle \cdot\mid \cdot\rangle \big)$,
we consider the semi-Hilbertian space $\big(\mathcal{H}, \langle \cdot\mid \cdot\rangle_A \big)$ where ${\langle x\mid y\rangle}_A := \langle Ax\mid y\rangle$ for every $x, y\in\mathcal{H}$. The $A$-numerical radius of an $A$-bounded operator $T$ on $\mathcal{H}$ is given by
\begin{align*}
\omega_A(T) = \sup\Big\{\big|{\langle Tx\mid x\rangle}_A\big|\,; \,\,x\in \mathcal{H}, \,{\langle x\mid x\rangle}_A= 1\Big\}.
\end{align*}
Our aim in this paper is to derive several $\mathbb{A}$-numerical radius inequalities for $2\times 2$ operator matrices whose entries are $A$-bounded operators, where $\mathbb{A}=\text{diag}(A,A)$.
\end{abstract}

\section{Introduction and Preliminaries}\label{s1}
Let $\mathcal{H}$ be a complex Hilbert space with inner product $\langle\cdot\mid\cdot\rangle$ and associated norm $\|\cdot\|$. Let $\mathcal{B}(\mathcal{H})$ stand for the $C^{\ast}$-algebra of all bounded linear operators on $\mathcal{H}$. The symbol $I$ denotes the identity operator on $\mathcal{H}$. Let $\mathcal{B}(\mathcal{H})^+$ be the cone of all positive (semi-definite) operators in $\mathcal{B}(\mathcal{H})$, i.e.,
$$\mathcal{B}(\mathcal{H})^+=\left\{A\in \mathcal{B}(\mathcal{H})\,;\;\langle Ax\mid x\rangle\geq 0,\;\forall\;x\in \mathcal{H}\,\right\}.$$
In all what follows, by an operator we mean a bounded linear operator. Moreover, for $T\in\mathcal{B}(\mathcal{H})$, we denote by $\mathcal{N}(T)$ and $\mathcal{R}(T)$ the kernel and the range of $T$, respectively. Furthermore, $T^*$ is the adjoint of $T$. For a given linear subspace $\mathcal{M}$ of $\mathcal{H}$, its closure in the norm topology of $\mathcal{H}$ will be denoted by $\overline{\mathcal{M}}$. In addition, let $P_{\mathcal{S}}$
stand for the orthogonal projection onto a closed subspace $\mathcal{S}$ of $\mathcal{H}$.

Let $A\in \mathcal{B}(\mathcal{H})^+$. Then, $A$ induces the following semi-inner product
$$\langle\cdot\mid\cdot\rangle_{A}:\mathcal{H}\times \mathcal{H}\longrightarrow\mathbb{C},\;(x,y)\longmapsto \langle x\mid y\rangle_{A}:=\langle Ax\mid y\rangle=\langle A^{1/2}x\mid A^{1/2}y\rangle.$$
Here $A^{1/2}$ stands for the square root of $A$. The seminorm induced by ${\langle \cdot\mid \cdot\rangle}_A$ is given by ${\|x\|}_A=\|A^{1/2}x\|$ for all $x\in\mathcal{H}$. One can verify that ${\|\cdot\|}_A$ is a norm if and only if $A$ is one-to-one, and that the seminormed space $(\mathcal{H}, {\|\cdot\|}_A)$ is complete if and only if $\overline{\mathcal{R}(A)}=\mathcal{R}(A)$. The semi-inner product ${\langle \cdot\mid \cdot\rangle}_{A}$ induces on the quotient $\mathcal{H}/\mathcal{N}(A)$ an inner product which is not complete unless $\mathcal{R}(A)$ is closed. However, a canonical construction due to de Branges and Rovnyak \cite{branrov} shows that the completion of $\mathcal{H}/\mathcal{N}(A)$ is isometrically isomorphic to the Hilbert space $\mathcal{R}(A^{1/2})$ endowed with the following inner product
\begin{align}\label{I.1}
\langle A^{1/2}x,A^{1/2}y\rangle_{\mathbf{R}(A^{1/2})}:=\langle P_{\overline{\mathcal{R}(A)}}x\mid P_{\overline{\mathcal{R}(A)}}y\rangle,\quad\forall\, x,y \in \mathcal{H}.
\end{align}
 For the sequel, the Hilbert space $\big(\mathcal{R}(A^{1/2}), \langle\cdot,\cdot\rangle_{\mathbf{R}(A^{1/2})}\big)$ will be denoted by $\mathbf{R}(A^{1/2})$. It is worth noting that $\mathcal{R}(A)$ is dense in $\mathbf{R}(A^{1/2})$ (see \cite{acg3}). For an account of results related to the Hilbert space $\mathbf{R}(A^{1/2})$, the reader is invited to consult \cite{acg3} and the references therein. By using \eqref{I.1}, it can be checked that
\begin{align}\label{I.2}
\langle Ax,Ay\rangle_{\mathbf{R}(A^{1/2})}= {\langle x, y\rangle}_{A},\quad\forall\, x,y \in \mathcal{H}.
\end{align}
Let $T \in \mathcal{B}(\mathcal{H})$. An operator $S\in\mathcal{B}(\mathcal{H})$ is said to be an $A$-adjoint of $T$ if for all $x,y\in \mathcal{H}$, the identity $\langle Tx\mid y\rangle_A=\langle x\mid Sy\rangle_A$ holds (see \cite{acg1}). So, the existence of an $A$-adjoint of $T$ is equivalent to the existence of a solution of the equation $AX = T^*A$. Notice that this kind of equations can be investigated by using a well-known theorem due to Douglas \cite{doug} which briefly says that the operator equation $TX=S$ has a bounded linear solution $X$ if and only if $\mathcal{R}(S) \subseteq \mathcal{R}(T)$ if and only if there exists a positive number $\lambda$ such that $\|S^*x\|\leq \lambda \|T^*x\|$ for all $x\in \mathcal{H}$. Furthermore, among its many solutions it has only one, denoted $Q$, which satisfies $\mathcal{R}(Q) \subseteq \overline{\mathcal{R}(T^{*})}$. Such $Q$ is called the Douglas solution or the reduced solution of the equation $TX=S$. Clearly, the existence of an $A$-adjoint operator is not guaranteed. If we denote by $\mathcal{B}_{A}(\mathcal{H})$ the subspace of all operators admitting $A$-adjoints, then by Douglas theorem, we  have
$$\mathcal{B}_{A}(\mathcal{H})=\left\{T\in \mathcal{B}(\mathcal{H})\,;\;\mathcal{R}(T^{*}A)\subset \mathcal{R}(A)\right\}.$$
If $T\in \mathcal{B}_A(\mathcal{H})$, the reduced solution of the equation
$AX=T^*A$ is a distinguished $A$-adjoint operator of $T$, which is denoted by $T^{\sharp_A}$. Note that, $T^{\sharp_A}=A^\dag T^*A$ in which $A^\dag$ is the Moore-Penrose inverse of $A$ (see \cite{acg2}). Notice that if $T \in \mathcal{B}_A({\mathcal{H}})$, then $T^{\sharp_A} \in \mathcal{B}_A({\mathcal{H}})$, $(T^{\sharp_A})^{\sharp_A}=P_{\overline{\mathcal{R}(A)}}TP_{\overline{\mathcal{R}(A)}}$ and $((T^{\sharp_A})^{\sharp_A})^{\sharp_A}=T$. Moreover, If $S\in \mathcal{B}_A(\mathcal{H})$ then $TS \in\mathcal{B}_A({\mathcal{H}})$ and $(TS)^{\sharp_A}=S^{\sharp_A}T^{\sharp_A}.$ In addition for every $T \in \mathcal{B}_A({\mathcal{H}})$ we have
\begin{equation}\label{diez}
\|T^{\sharp_A}T\|_A = \| TT^{\sharp_A}\|_A=\|T\|_A^2 =\|T^{\sharp_A}\|_A^2.
\end{equation}
For results concerning $T^{\sharp_A}$, we refer the reader to \cite{acg1,acg2}. An operator $U\in  \mathcal{B}_A(\mathcal{H})$ is called $A$-unitary if $\|Ux\|_A=\|U^{\sharp_A}x\|_A=\|x\|_A$ for all $x\in \mathcal{H}$. It should be mention that, an operator $U\in  \mathcal{B}_A(\mathcal{H})$ is $A$-unitary if and only if $U^{\sharp_A} U=(U^{\sharp_A})^{\sharp_A} U^{\sharp_A}=P_{\overline{\mathcal{R}(A)}}$ (see \cite{acg1}).

An operator $T$ is called $A$-bounded if there exists $\lambda>0$ such that $ \|Tx\|_{A} \leq \lambda \|x\|_{A},\;\forall\,x\in \mathcal{H}.$  An application of Douglas theorem shows that the subspace of all operators admitting $A^{1/2}$-adjoints, denoted by $\mathcal{B}_{A^{1/2}}(\mathcal{H})$, is equal the collection of all $A$-bounded operators, i.e.,
$$\mathcal{B}_{A^{1/2}}(\mathcal{H})=\left\{T \in \mathcal{B}(\mathcal{H})\,;\;\exists \,\lambda > 0\,;\;\|Tx\|_{A} \leq \lambda \|x\|_{A},\;\forall\,x\in \mathcal{H}  \right\}.$$
Notice that $\mathcal{B}_{A}(\mathcal{H})$ and $\mathcal{B}_{A^{1/2}}(\mathcal{H})$ are two subalgebras of $\mathcal{B}(\mathcal{H})$ which are, in general, neither closed nor dense in $\mathcal{B}(\mathcal{H})$. Moreover, we have $\mathcal{B}_{A}(\mathcal{H})\subset \mathcal{B}_{A^{1/2}}(\mathcal{H})$ (see \cite{acg1,acg3}). Clearly, $\langle\cdot\mid\cdot\rangle_{A}$ induces a seminorm on $\mathcal{B}_{A^{1/2}}(\mathcal{H})$. Indeed, if $T\in\mathcal{B}_{A^{1/2}}(\mathcal{H})$, then it holds that
\begin{equation}\label{semii}
\|T\|_A:=\sup_{\substack{x\in \overline{\mathcal{R}(A)},\\ x\not=0}}\frac{\|Tx\|_A}{\|x\|_A}=\sup\big\{{\|Tx\|}_A\,; \,\,x\in \mathcal{H},\, {\|x\|}_A =1\big\}<\infty.
\end{equation}
Notice that it was proved in \cite{fg} that for $T\in\mathcal{B}_{A^{1/2}}(\mathcal{H})$ we have
\begin{equation}\label{newsemi}
\|T\|_A=\sup\left\{|\langle Tx\mid y\rangle_A|\,;\;x,y\in \mathcal{H},\,\|x\|_{A}=\|y\|_{A}= 1\right\}.
\end{equation}
Furthermore, the $A$-numerical radius of an operator $T\in\mathcal{B}(\mathcal{H})$ was firstly defined by Saddi in \cite{saddi} by
\begin{align*}
\omega_A(T)
&:= \sup \left\{|\langle Tx\mid x\rangle_A|\,;\;x\in\mathcal{H},\|x\|_A = 1\right\}.
  \end{align*}
It should be emphasized that it may happen that ${\|T\|}_A$ and $\omega_A(T)$ are equal to $+ \infty$ for some $T\in\mathcal{B}(\mathcal{H})\setminus \mathcal{B}_{A^{1/2}}(\mathcal{H})$ (see \cite{feki01}). However, these quantities are equivalent seminorms on $\mathcal{B}_{A^{1/2}}(\mathcal{H})$. More precisely, it was shown in \cite{bakfeki01} that for every $T\in \mathcal{B}_{A^{1/2}}(\mathcal{H})$, we have
\begin{equation}\label{refine1}
\tfrac{1}{2} \|T\|_A\leq\omega_A(T) \leq \|T\|_A.
\end{equation}
Notice that if $T\in\mathcal{B}_{A^{1/2}}(\mathcal{H})$ and satisfies $AT^2=0$, then by \cite[Corollary 2]{feki01} we have
\begin{equation}\label{at2}
\omega_A(T)= \frac{1}{2}\|T\|_A.
\end{equation}
Notice that the $A$-numerical radius of semi-Hilbertian space operators satisfies the weak $A$-unitary invariance property which asserts that
\begin{equation}\label{weak}
\omega_{A}(U^{\sharp}TU)=\omega_{A}(T),
\end{equation}
for every $T\in \mathcal{B}_{A^{1/2}}(\mathcal{H})$ and every $A$-unitary operator $U\in \mathcal{B}_{A}(\mathcal{H})$ (see \cite[Lemma 3.8]{bfeki}).

For the sequel, for any arbitrary operator $T\in {\mathcal B}_A({\mathcal H})$, we write
$$\Re_A(T):=\frac{T+T^{\sharp_A}}{2}\;\;\text{ and }\;\;\Im_A(T):=\frac{T-T^{\sharp_A}}{2i}.$$
It has recently been shown in \cite[Theorem 2.5]{zamani1} that if $T\in\mathcal{B}_{A}(\mathcal{H})$, then
\begin{align}\label{zm}
\omega_A(T) = \displaystyle{\sup_{\theta \in \mathbb{R}}}{\left\|\Re_A(e^{i\theta}T)\right\|}_A=\displaystyle{\sup_{\theta \in \mathbb{R}}}{\left\|\Im_A(e^{i\theta}T)\right\|}_A.
\end{align}
	Let $T\in \mathcal{B}(\mathcal{H})$. Then, it was shown in \cite[Proposition 3.6.]{acg3} that $T\in \mathcal{B}_{A^{1/2}}(\mathcal{H})$ if and only if there exists a unique $\widetilde{T}\in \mathcal{B}(\mathbf{R}(A^{1/2}))$ such that $Z_AT =\widetilde{T}Z_A$. Here, $Z_{A}: \mathcal{H} \rightarrow \mathbf{R}(A^{1/2})$ is	defined by $Z_{A}x = Ax$. It has been shown in \cite{feki01} that for every $T\in \mathcal{B}_{A^{1/2}}(\mathcal{H})$ we have
\begin{equation}\label{tilde}
\|T\|_A=\|\widetilde{T}\|_{\mathcal{B}(\mathbf{R}(A^{1/2}))}\quad\text{ and }\quad \omega_A(T)=\omega(\widetilde{T}).
\end{equation}

Recently, the concept of the $A$-spectral radius of $A$-bounded operators has been introduced in \cite{feki01} as follows:
\begin{equation}\label{newrad}
r_A(T):=\displaystyle\inf_{n\geq 1}\|T^n\|_A^{\frac{1}{n}}=\displaystyle\lim_{n\to\infty}\|T^n\|_A^{\frac{1}{n}}.
\end{equation}
We note here that the second equality in \eqref{newrad} is also proved in \cite[Theorem 1]{feki01}. Moreover, like the classical spectral radius of Hilbert space operators, it was shown in \cite{feki01} that $r_A(\cdot)$ satisfies the commutativity property, which asserts that
\begin{equation}\label{commut}
r_A(TS)=r_A(ST),
\end{equation}
for all $T,S\in \mathcal{B}_{A^{1/2}}(\mathcal{H})$.

An operator $T\in\mathcal{B}(\mathcal{H})$ is said to be $A$-selfadjoint if $AT$ is selfadjoint, that is, $AT = T^*A$. Moreover, it was shown in \cite{feki01} that if $T$ is $A$-self-adjoint, then
\begin{equation}\label{aself1}
\|T\|_{A}=\omega_A(T)=r_A(T).
\end{equation}
In addition, an operator $T$ is called $A$-positive if $AT\geq0$ and we write $T\geq_{A}0$. Obviously, an $A$-positive operator is always an $A$-selfadjoint operator since $\mathcal{H}$ is a complex Hilbert space. If $T,S\in\mathcal{B}(\mathcal{H})$ and satisfies $T-S\geq_{A}0$, then we will write $T\geq_{A}S$. For the sequel, if $A=I$ then $\|T\|$, $r(T)$ and $\omega(T)$ denote respectively the classical operator norm, the spectral radius and the numerical radius of an operator $T$. In recent years, several results covering some classes of operators on a complex Hilbert space $\big(\mathcal{H}, \langle \cdot\mid \cdot\rangle\big)$ were extended to $\big(\mathcal{H}, {\langle \cdot\mid \cdot\rangle}_A\big)$. Of course, the extension is not trivial since many difficulties arise. For instance, as it is mention above, it may happen
that ${\|T\|}_A = \infty$ for some $T\in \mathcal{B}(\mathcal{H})$. Moreover, not any operator admits an adjoint operator for the
semi-inner product ${\langle \cdot\mid \cdot\rangle}_A$. In addition, for $T \in \mathcal{B}_A({\mathcal{H}})$ we have $(T^{\sharp_A})^{\sharp_A}=P_{\overline{\mathcal{R}(A)}}TP_{\overline{\mathcal{R}(A)}}\neq T$. The reader is invited to see \cite{bakfeki01,bakfeki04,bfeki,feki03,zamani2,tamzhang,zamani1,zamani3} and the references therein.

In this paper, we consider the ${2\times 2}$ operator diagonal matrix $\mathbb{A}=\begin{pmatrix}
A &0\\
0 &A
\end{pmatrix}$. Clearly, $\mathbb{A}\in \mathcal{B}(\mathcal{H}\oplus \mathcal{H})^+$. So, $\mathbb{A}$ induces the following semi-inner product
$$\langle x, y\rangle_{\mathbb{A}}= \langle \mathbb{A}x, y\rangle=\langle x_1\mid y_1\rangle_A+\langle x_2\mid y_2\rangle_A,$$
 for all $x=(x_1,x_2)\in \mathcal{H}\oplus \mathcal{H}$ and $y=(y_1,y_2)\in \mathcal{H}\oplus \mathcal{H}$. Notice that if $T_{ij}$ are operators in $\mathcal{B}_{A}(\mathcal{H})$ for all $i,j\in\{1,2\}$. Then, it was shown in \cite[Lemma 3.1]{bfeki} that $(T_{ij})_{2 \times 2}\in \mathcal{B}_{\mathbb{A}}(\mathcal{H}\oplus \mathcal{H})$ and
\begin{equation}\label{diez2}
\begin{pmatrix}
T_{11}&T_{12} \\
T_{21}&T_{22}
\end{pmatrix}^{\sharp_\mathbb{A}}=\begin{pmatrix}
T^{\sharp_A}_{11} & T^{\sharp_A}_{21} \\
T^{\sharp_A}_{12} & T^{\sharp_A}_{22}
\end{pmatrix}.
\end{equation}
 Very recently, several inequalities for the $\mathbb{A}$-numerical radius of $2 \times 2$ operator matrices have been established by P. Bhunia et al. (see \cite{BPN}). This paper is devoted also to prove several new $\mathbb{A}$-numerical radius inequalities of certain $2 \times 2$ operator matrices. Some of the obtained results cover and extend the following works \cite{pinar,HirKit,S}.

 \section{Results}\label{s2}
In this section, we present our results. Throughout this section $\mathbb{A}$ is denoted to be the $2\times 2$ operator diagonal matrix whose each diagonal entry is the positive operator $A$. To prove our two next results, the following lemma concerning $\mathbb{A}$-numerical radius inequalities is required. Notice that the first assertion is proved in \cite{BPN} for operators in $\mathcal{B}_A(\mathcal{H})$.

\begin{lemma}\label{lem01}
Let $P,Q,R,S\in \mathcal{B}_{A^{1/2}}(\mathcal{H})$. Then, the following assertions hold:
\begin{itemize}
  \item [(a)] $\omega_{\mathbb{A}}\left[\begin{pmatrix}
P&0\\
0 &S
\end{pmatrix}\right]=\max\{\omega_A(P),\omega_A(S)\}.$
  \item [(b)] $\omega_{\mathbb{A}}\left[\begin{pmatrix}
P &0\\
0 &S
\end{pmatrix}\right]\leq\omega_{\mathbb{A}}\left[\begin{pmatrix}
P &Q\\
R &S
\end{pmatrix}\right].$
  \item [(c)] $\omega_{\mathbb{A}}\left[\begin{pmatrix}
0 &Q\\
R &0
\end{pmatrix}\right]\leq\omega_{\mathbb{A}}\left[\begin{pmatrix}
P &Q\\
R &S
\end{pmatrix}\right].$
\end{itemize}
\end{lemma}
\begin{proof}
\noindent (a)\;Follows by proceeding as in the proof of \cite[Lemma 2.4.]{BPN}.
\par \vskip 0.1 cm \noindent (b)\;Clearly we have
\begin{equation}\label{pinching}
 \begin{pmatrix}
P &0\\
0 &S
\end{pmatrix} = \frac{1}{2} \begin{pmatrix}
P &Q\\
R &S
\end{pmatrix} + \frac{1}{2} \begin{pmatrix}
P &-Q\\
-R &S
\end{pmatrix}.
\end{equation}
Let $\mathbb{U}=\begin{pmatrix}
-I&O \\
O&I
\end{pmatrix}.$ In view of \eqref{diez2} we have $\mathbb{U}^{\sharp_{\mathbb{A}}}=\begin{pmatrix}
-P_{\overline{\mathcal{R}(A)}}&O \\
O&P_{\overline{\mathcal{R}(A)}}
\end{pmatrix}.$ So, we verify that $\|\mathbb{U}x\|_\mathbb{A}=\|\mathbb{U}^{\sharp_\mathbb{A}}x\|_\mathbb{A}=\|x\|_\mathbb{A}$ for all $x=(x_1,x_2)\in \mathcal{H}\oplus \mathcal{H}$. Hence, $\mathbb{U}$ is $\mathbb{A}$-unitary operator. Thus, by \eqref{weak} we have
\begin{align*}
\omega_{\mathbb{A}}\left[\begin{pmatrix}
P &Q\\
R &S
\end{pmatrix}\right]
& =\omega_\mathbb{A}\left[\mathbb{U}^{\sharp_A}\begin{pmatrix}
P &Q\\
R &S
\end{pmatrix}\mathbb{U}\right]\\
& =\omega_\mathbb{A}\left[\mathbb{U}^{\sharp_A}\begin{pmatrix}
P &Q\\
R &S
\end{pmatrix}\mathbb{U}\right]\\
& =\omega_\mathbb{A}\left[\begin{pmatrix}
P_{\overline{\mathcal{R}(A)}}&O \\
O&P_{\overline{\mathcal{R}(A)}}
\end{pmatrix}\begin{pmatrix}
P &-Q\\
-R &S
\end{pmatrix}\right]\\
&=\omega_{\mathbb{A}}\left[\begin{pmatrix}
P &-Q\\
-R &S
\end{pmatrix}\right]
\end{align*}
So, by taking into consideration \eqref{pinching} and the triangle inequality we prove the desired result.
\par \vskip 0.1 cm \noindent (b)\;Let $\mathbb{U}=\begin{pmatrix}
I&O \\
O&-I
\end{pmatrix}.$ By proceeding similarly as above, we prove that $\mathbb{U}$ is $\mathbb{A}$-unitary and
$$\omega_\mathbb{A}\left[\begin{pmatrix}
P &Q\\
R &S
\end{pmatrix}\right]=\omega_\mathbb{A}\left[\begin{pmatrix}
-P &Q\\
R &-S
\end{pmatrix}\right].$$
Moreover, by using the fact that
\begin{equation*}
 \begin{pmatrix}
0 &Q\\
R &0
\end{pmatrix} = \frac{1}{2} \begin{pmatrix}
P &Q\\
R &S
\end{pmatrix} + \frac{1}{2} \begin{pmatrix}
-P &Q\\
R &-S
\end{pmatrix},
\end{equation*}
and the subadditivity of the $\mathbb{A}$-numerical radius $\omega_{\mathbb{A}}(\cdot)$, we get the required result.
\end{proof}
Also, we need the following lemma.
 \begin{lemma}(\cite{feki04})\label{max}
Let $T,S\in \mathcal{B}_{A^{1/2}}(\mathcal{H})$. Then,
$$\left\|\begin{pmatrix}
0&T\\
S &0
\end{pmatrix}\right\|_{\mathbb{A}}=\left\|\begin{pmatrix}
T&0\\
0 &S
\end{pmatrix}\right\|_{\mathbb{A}}=\max\{\|T\|_A,\|S\|_A \}.$$
\end{lemma}
Now, we are in a position to prove our first result in this paper.
 \begin{theorem}
Let $\mathbb{T}=\begin{pmatrix}
P &Q\\
R &S
\end{pmatrix}$ be such that $P, Q, R, S\in \mathcal{B}_{A^{1/2}}(\mathcal{H})$. Then,
\begin{equation}
\lambda_1\leq\omega_{\mathbb{A}}\left[\begin{pmatrix}
P &Q\\
R &S
\end{pmatrix}\right]\leq \lambda_2,
\end{equation}
where $$\lambda_1=\max\left\{\omega_{\mathbb{A}}\left[\begin{pmatrix}
0 &Q\\
R &0
\end{pmatrix}\right],\max\{\omega_A(P),\omega_A(S)\}\right\}$$ and
$$\lambda_2=\frac{\|Q\|_A+\|R\|_A}{2}+\max\big\{\omega_A(P),\omega_A(S)\big\}.$$
\end{theorem}
\begin{proof}
Clearly we have
\begin{equation}\label{triv}
 \begin{pmatrix}
P &Q\\
R &S
\end{pmatrix} = \begin{pmatrix}
P &0\\
0 &S
\end{pmatrix}+\begin{pmatrix}
0 &Q\\
0 &0
\end{pmatrix}+\begin{pmatrix}
0 &0\\
R &0
\end{pmatrix}.
\end{equation}
On the other, it is not difficult to see that $\mathbb{A}\begin{pmatrix}
0 &Q\\
0 &0
\end{pmatrix}^2=\begin{pmatrix}
0 &0\\
0 &0
\end{pmatrix}$ and $\mathbb{A}\begin{pmatrix}
0 &0\\
R &0
\end{pmatrix}^2=\begin{pmatrix}
0 &0\\
0 &0
\end{pmatrix}$. So, by \eqref{at2} and Lemma \ref{max} we have
$$\omega_{\mathbb{A}}\left[\begin{pmatrix}
0 &Q\\
0 &0
\end{pmatrix}\right]=\frac{1}{2}\left\|\begin{pmatrix}
0 &Q\\
0 &0
\end{pmatrix}\right\|_{\mathbb{A}}=\frac{1}{2}\|Q\|_A.$$
Similarly, we have $\omega_{\mathbb{A}}\left[\begin{pmatrix}
0 &0\\
R &0
\end{pmatrix}\right]=\frac{1}{2}\|R\|_A.$ So, by using the trivial observation \eqref{triv} and the subadditivity of the $\mathbb{A}$-numerical radius $\omega_{\mathbb{A}}(\cdot)$ together with Lemma \ref{lem01} (a), we get
\begin{equation}\label{r1}
\omega_{\mathbb{A}}\left[\begin{pmatrix}
P &Q\\
R &S
\end{pmatrix}\right]\leq \max\{\omega_A(P),\omega_A(S)\}+\frac{\|Q\|_A+\|R\|_A}{2}.
\end{equation}
On the other hand, by Lemmas \ref{lem01} and \ref{max} we have
\begin{align}\label{r2}
\omega_{\mathbb{A}}\left[\begin{pmatrix}
P &Q\\
R &S
\end{pmatrix}\right]
& \geq \max\left\{\omega_{\mathbb{A}}\left[\begin{pmatrix}
0 &Q\\
R &0
\end{pmatrix}\right],\omega_{\mathbb{A}}\left[\begin{pmatrix}
P &0\\
0 &S
\end{pmatrix}\right]\right\} \nonumber\\
 &=\max\left\{\omega_{\mathbb{A}}\left[\begin{pmatrix}
0 &Q\\
R &0
\end{pmatrix}\right],\max\{\omega_A(P),\omega_A(S)\}\right\}.
\end{align}
By combining \eqref{r1} together with \eqref{r2}, we reach the desired result.
\end{proof}

In order to prove our next result, we need the following lemma.
\begin{lemma}\label{lpos}
Let $T,S\in \mathcal{B}(\mathcal{H})$ be two $A$-positive operators. Then,
\begin{equation}
\omega_{\mathbb{A}}\left[\begin{pmatrix}0&T\\S&0\end{pmatrix}\right]= \frac{1}{2}\left\|T+S\right\|_A.
\end{equation}
\end{lemma}
\begin{proof}
Since $T$ and $S$ are $A$-positive, then $T,S\in \mathcal{B}_{A^{1/2}}(\mathcal{H})$. So, by \cite[Proposition 3.6.]{acg3} there exists two unique operators $\widetilde{T},\widetilde{S}\in \mathcal{B}(\mathbf{R}(A^{1/2}))$ such that $Z_AT =\widetilde{T}Z_A$ and $Z_AS =\widetilde{S}Z_A$. Moreover, since $T\geq_A 0$, then for all $x\in \mathcal{H}$ we have
\begin{equation*}
\langle ATx\mid x\rangle\geq 0.
\end{equation*}
This implies, through \eqref{I.2}, that
\begin{align*}
\langle Tx\mid x\rangle_A=\langle ATx,Ax\rangle_{\mathbf{R}(A^{1/2})}=\langle \widetilde{T}Ax,Ax\rangle_{\mathbf{R}(A^{1/2})}\geq 0
\end{align*}
for all $x\in \mathcal{H}$. Further, by using the density of $\mathcal{R}(A)$ in $\mathbf{R}(A^{1/2})$, we obtain
\begin{equation*}
\langle \widetilde{T}A^{1/2}x,A^{1/2}x\rangle_{\mathbf{R}(A^{1/2})}\geq0,\;\;\forall\,x\in \mathcal{H}.
\end{equation*}
So, $\widetilde{T}$ is a positive operator on the Hilbert space $\mathbf{R}(A^{1/2})$. Similarly, we prove that $\widetilde{S}\geq 0$. Therefore, in view of \cite[Corollary3.]{OK2} we have
\begin{equation}\label{one}
\omega\left[\begin{pmatrix}0&\widetilde{T}\\ \widetilde{S}&0\end{pmatrix}\right]= \frac{1}{2}\left\|\widetilde{T}+\widetilde{S}\right\|_{\mathcal{B}(\mathbf{R}(A^{1/2}))}=\frac{1}{2}\left\|\widetilde{T+S}\right\|_{\mathcal{B}(\mathbf{R}(A^{1/2}))},
\end{equation}
where the last equality follows since $\widetilde{T+S}=\widetilde{T}+\widetilde{S}$. Moreover, by \cite[Lemma 3.2]{bfeki}, we have $\begin{pmatrix}0&T\\ S&0\end{pmatrix}\in \mathcal{B}_{\mathbb{A}^{1/2}}(\mathcal{H}\oplus \mathcal{H})$ and
$$ \widetilde{\begin{pmatrix}0&T\\ S&0\end{pmatrix}}=\begin{pmatrix}0&\widetilde{T}\\ \widetilde{S}&0\end{pmatrix}.$$
This proves the desired result by applying \eqref{one} together with \eqref{tilde}.
\end{proof}

We are now in a position to state the following theorem.
\begin{theorem}\label{thmf}
Let $\mathbb{T}=\begin{pmatrix}
P &Q\\
R &S
\end{pmatrix}$ be such that $P, Q, R, S\in \mathcal{B}_{A}(\mathcal{H})$. Then,
\begin{align*}
\omega_{\mathbb{A}}(\mathbb{T})
&\leq \frac{1}{2}\Big(\omega_A(P)+\omega_A(Q)\Big)+\frac{1}{4}\Big(\|I+PP^{\sharp_A}+QQ^{\sharp_A}\|_A+\|I+RR^{\sharp_A}+SS^{\sharp_A}\|_A\Big).
\end{align*}
\end{theorem}
\begin{proof}
We first prove that
\begin{equation}\label{f01}
\omega_{\mathbb{A}}(\mathbb{S})\leq \tfrac{1}{2}\omega_A(P)+\tfrac{1}{4}\|I+PP^{\sharp_A}+QQ^{\sharp_A}\|_A,
\end{equation}
where $\mathbb{S}=\begin{pmatrix}
P &Q\\
0 &0
\end{pmatrix}$. Let $\theta\in \mathbb{R}$. It is not difficult to verify that $\Re_\mathbb{A}(e^{i\theta}\mathbb{S})$ is an $\mathbb{A}$-self-adjoint operator. So, by \eqref{aself1} we have
$$r_{\mathbb{A}}\left(\Re_\mathbb{A}(e^{i\theta}\mathbb{S})\right)=\|\Re_\mathbb{A}(e^{i\theta}\mathbb{S})\|_{\mathbb{A}}.$$
Now, by using \eqref{diez2}, we see that
\begin{align*}
r_{\mathbb{A}}\left[\Re_A(e^{i\theta}\mathbb{S})\right]
& =\tfrac{1}{2}r_{\mathbb{A}}(e^{i\theta}\mathbb{S}+e^{-i\theta}\mathbb{S}^{\sharp_{\mathbb{A}}})\\
 &=\tfrac{1}{2}r_{\mathbb{A}}\left[e^{i\theta}\begin{pmatrix}
P &Q\\
0 &0
\end{pmatrix}+e^{-i\theta}\begin{pmatrix}
P^{\sharp_{\mathbb{A}}} &0\\
Q^{\sharp_{\mathbb{A}}} &0
\end{pmatrix}\right]\\
&=\tfrac{1}{2}r_{\mathbb{A}}\left[\begin{pmatrix}
e^{i\theta}P+e^{-i\theta}P^{\sharp_{A}} &e^{i\theta}Q\\
e^{-i\theta}Q^{\sharp_{A}} &0
\end{pmatrix}\right]\\
&=\tfrac{1}{2}r_{\mathbb{A}}\left[
\begin{pmatrix}
P^{\sharp_{A}} &e^{i\theta}I\\
Q^{\sharp_{A}} &0
\end{pmatrix}
\begin{pmatrix}
e^{-i\theta}I &0\\
P &Q
\end{pmatrix}
\right]\\
&=\tfrac{1}{2}r_{\mathbb{A}}\left[
\begin{pmatrix}
e^{-i\theta}I &0\\
P &Q
\end{pmatrix}
\begin{pmatrix}
P^{\sharp_{A}} &e^{i\theta}I\\
Q^{\sharp_{A}} &0
\end{pmatrix}
\right]\quad (\text{by }\; \eqref{commut})\\
&=\tfrac{1}{2}r_{\mathbb{A}}\left[
\begin{pmatrix}
e^{-i\theta}P^{\sharp_{A}} &I\\
PP^{\sharp_{A}}+QQ^{\sharp_{A}} &e^{i\theta}P
\end{pmatrix}
\right]\\
&\leq\tfrac{1}{2}\omega_{\mathbb{A}}\left[
\begin{pmatrix}
e^{-i\theta}P^{\sharp_{A}} &I\\
PP^{\sharp_{A}}+QQ^{\sharp_{A}} &e^{i\theta}P
\end{pmatrix}
\right]\\
&\leq\tfrac{1}{2}\omega_{\mathbb{A}}\left[
\begin{pmatrix}
e^{-i\theta}P^{\sharp_{A}} &0\\
0 &e^{i\theta}P
\end{pmatrix}
\right]+\tfrac{1}{2}\omega_{\mathbb{A}}\left[
\begin{pmatrix}
0 &I\\
PP^{\sharp_{A}}+QQ^{\sharp_{A}} &0
\end{pmatrix}
\right]\\
&=\omega_A(P)+\tfrac{1}{2}\|I+PP^{\sharp_{A}}+QQ^{\sharp_{A}}\|_A,\;(\text{by Lemmas } \ref{lem01} \text{ and }\ref{lpos}).
\end{align*}
Hence,
$$
\|\Re_\mathbb{A}(e^{i\theta}\mathbb{S})\|_{\mathbb{A}}\leq\omega_A(P)+\tfrac{1}{2}\|I+PP^{\sharp_{A}}+QQ^{\sharp_{A}}\|_A.$$
So, by taking the supremum over all $\theta\in \mathbb{R}$ and then applying \eqref{zm} we obtain \eqref{f01} as required. Let $\mathbb{U}=\begin{pmatrix}
0&I \\
I&0
\end{pmatrix}.$ In view of \eqref{diez2} we have $\mathbb{U}^{\sharp_{\mathbb{A}}}=\begin{pmatrix}
0&P_{\overline{\mathcal{R}(A)}}\\
P_{\overline{\mathcal{R}(A)}}&0
\end{pmatrix}.$ Further, it can be seen that $\mathbb{U}$ is $\mathbb{A}$-unitary operator. So, by using \eqref{weak} together with \eqref{f01} we get
\begin{align*}
\omega_{\mathbb{A}}(\mathbb{T})
& \leq\omega_{\mathbb{A}}\left[\begin{pmatrix}
P &Q\\
0 &0
\end{pmatrix}\right]+\omega_{\mathbb{A}}\left[\begin{pmatrix}
0 &0\\
R &S
\end{pmatrix}\right] \\
 &=\omega_{\mathbb{A}}\left[\begin{pmatrix}
P &Q\\
0 &0
\end{pmatrix}\right]+\omega_{\mathbb{A}}\left[\mathbb{U}^{\sharp_{\mathbb{A}}}\begin{pmatrix}
0 &0\\
R &S
\end{pmatrix}\mathbb{U}\right] \\
 &=\omega_{\mathbb{A}}\left[\begin{pmatrix}
P &Q\\
0 &0
\end{pmatrix}\right]+\omega_{\mathbb{A}}\left[\begin{pmatrix}
P_{\overline{\mathcal{R}(A)}} &0\\
0 &P_{\overline{\mathcal{R}(A)}}
\end{pmatrix}
\begin{pmatrix}
S &R\\
0 &0
\end{pmatrix}\right]\\
 &=\omega_{\mathbb{A}}\left[\begin{pmatrix}
P &Q\\
0 &0
\end{pmatrix}\right]+\omega_{\mathbb{A}}\left[
\begin{pmatrix}
S &R\\
0 &0
\end{pmatrix}\right]\\
 &\leq\tfrac{1}{2}\left(\omega_A(P)+\omega_A(S)+\|PP^{\sharp_A}+QQ^{\sharp_A}\|_A^{1/2}+\|RR^{\sharp_A}+SS^{\sharp_A}\|_A^{1/2}\right).
\end{align*}
This finishes the proof of the theorem.
\end{proof}

The following lemma is useful in proving our next result.
\begin{lemma}\label{lm5}(\cite{feki02})
Let $\mathbb{T}=\begin{pmatrix}
T_{11}&T_{12} \\
T_{21}&T_{22}
\end{pmatrix}$ be such that $T_{ij}\in \mathcal{B}_{A^{1/2}}(\mathcal{H})$ for all $i,j\in\{1,2\}$. Then, $\mathbb{T}\in \mathcal{B}_{\mathbb{A}^{1/2}}(\mathcal{H}\oplus \mathcal{H})$ and
$$ r_\mathbb{A}\left(\mathbb{T}\right)\leq r\left[\begin{pmatrix}
\|T_{11}\|_A & \|T_{12}\|_A \\
\|T_{21}\|_A & \|T_{22}\|_A
\end{pmatrix}\right].$$
\end{lemma}
\begin{theorem}
Let $\mathbb{T}=\begin{pmatrix}
P &Q\\
R &S
\end{pmatrix}$ be such that $P, Q, R, S\in \mathcal{B}_{A}(\mathcal{H})$. Then,
\begin{equation}
\omega_{\mathbb{A}}(\mathbb{T})\leq \frac{1}{2}\left(\|P\|_A+\|S\|_A+\|PP^{\sharp_A}+QQ^{\sharp_A}\|_A^{1/2}+\|RR^{\sharp_A}+SS^{\sharp_A}\|_A^{1/2}\right).
\end{equation}
\end{theorem}
\begin{proof}
We first prove that
\begin{equation}
\omega_{\mathbb{A}}(\mathbb{S})\leq \frac{1}{2}\left(\|P\|_A+\|PP^{\sharp_A}+QQ^{\sharp_A}\|_A^{1/2}\right),
\end{equation}
where $\mathbb{S}=\begin{pmatrix}
P &Q\\
0 &0
\end{pmatrix}$. Let $\theta\in \mathbb{R}$. By proceeding as in the proof of Theorem \ref{thmf} we see that
\begin{align*}
\|\Re_\mathbb{A}(e^{i\theta}\mathbb{S})\|_{\mathbb{A}}
&=r_{\mathbb{A}}\left[\Re_\mathbb{A}(e^{i\theta}\mathbb{S})\right]\\
&=\tfrac{1}{2}r_{\mathbb{A}}\left[
\begin{pmatrix}
e^{-i\theta}P^{\sharp_{A}} &I\\
PP^{\sharp_{A}}+QQ^{\sharp_{A}} &e^{i\theta}P
\end{pmatrix}
\right]\\
&\leq\frac{1}{2}r\left[
\begin{pmatrix}
\|P\|_A &1\\
\|PP^{\sharp_{A}}+QQ^{\sharp_{A}}\|_A &\|P\|_A
\end{pmatrix}
\right]\\
&=\tfrac{1}{2}\left(\|P\|_A+\|PP^{\sharp_{A}}+QQ^{\sharp_{A}}\|_A^{1/2} \right).
\end{align*}
Using an argument similar to that used in proof of Theorem \ref{thmf}, we get the desired result.
\end{proof}

Before proving our next theorem we have to state the following lemma.
\begin{lemma}(\cite[Theorem 5.1]{bakfeki04})\label{pos2n}
Let $T \in \mathcal{B}(\mathcal{H})$ be an $A$-selfadjoint operator. Then, for any positive integer $n$ we have
\begin{equation*}
\|T^n\|_A=\|T\|_A^n.
\end{equation*}
\end{lemma}

 \begin{theorem}\label{theorem:upper 2}
Let $\mathbb{T}=\begin{pmatrix}
P &Q\\
R &S
\end{pmatrix}$ be such that $P, Q, R, S\in \mathcal{B}_{A}(\mathcal{H})$. Then,
 \[\omega_\mathbb{A}(\mathbb{T})\leq \sqrt{\omega_A^2(P)+\frac{1}{2}\|Q\|_A\left(\omega_A(P)+\frac{1}{2}\|Q\|_A\right)}+\sqrt{\omega_A^2(S)+\frac{1}{2}\|R\|_A\left(\omega_A(S)+\frac{1}{2}\|R\|_A\right)}.\]
\end{theorem}
\begin{proof}
Let $\mathbb{S}=\begin{pmatrix}
P&Q\\
0&0
\end{pmatrix}.$ We first prove that
\begin{equation}\label{ff1}
\omega_\mathbb{A}(\mathbb{S})\leq \sqrt{\omega_A^2(P)+\frac{1}{2}\|Q\|_A\left(\omega_A(P)+\frac{1}{2}\|Q\|_A\right)}.
\end{equation}
Let $\theta \in \mathbb{R}$. A straightforward calculation shows that
\begin{align*}
\Re_\mathbb{A}(e^{i\theta}\mathbb{S})
&=\begin{pmatrix}
\Re_A(e^{i\theta}P)&\frac{1}{2}e^{i\theta}Q\\
\frac{1}{2}e^{-i\theta}Q^{\sharp_A}&0
\end{pmatrix}\\
&=\begin{pmatrix}
\Re_A(e^{i\theta}P)&0\\
0&0
\end{pmatrix}+\begin{pmatrix}
0&\frac{1}{2}e^{i\theta}Q\\
\frac{1}{2}e^{-i\theta}Q^{\sharp_A}&0
\end{pmatrix}.
\end{align*}
This implies that
\begin{align*}
\big(\Re_\mathbb{A}(e^{i\theta}\mathbb{S})\big)^2
&=\begin{pmatrix}
[\Re_A(e^{i\theta}P)]^2&0\\
0&0
\end{pmatrix}+\begin{pmatrix}
\frac{1}{4}QQ^{\sharp_A}&0\\
0&\frac{1}{4}Q^{\sharp_A}Q
\end{pmatrix}\\
&+\begin{pmatrix}
0&\frac{1}{2}e^{i\theta} \left[\Re_A (e^{i\theta}P)\right]Q\\
0&0
\end{pmatrix}+\begin{pmatrix}
0&0\\
\frac{1}{2}e^{-i\theta}Q^{\sharp_A} \left[\Re_A (e^{i\theta}P)\right]&0
\end{pmatrix}.
\end{align*}
Thus, by using \eqref{zm} together with Lemma \ref{max} we see that
\begin{align*}
\|\big(\Re_\mathbb{A}(e^{i\theta}\mathbb{S})\big)^2\|_\mathbb{A}
&\leq \|\Re_A(e^{i\theta}P)\|_A^2+\frac{1}{4}\|Q\|_A^2+\frac{1}{2}\|\Re_A(e^{i\theta}P)\|_A \|Q\|_A\\
            &\leq \omega_A^2(P)+\frac{1}{4}\|Q\|_A^2+\frac{1}{2}\omega_A(P) \|Q\|_A.
\end{align*}
Since $\Re_\mathbb{A}(e^{i\theta}\mathbb{S})$ is $\mathbb{A}$-selfadjoint, then an application of Lemma \ref{pos2n} gives
$$
\|\Re_\mathbb{A}(e^{i\theta}\mathbb{S})\|_\mathbb{A}^2\leq \omega_A^2(P)+\frac{1}{4}\|Q\|_A^2+\frac{1}{2}\omega_A(P) \|Q\|_A.$$
Taking the supremum over all $\theta\in \mathbb{R}$ in the above inequality and then using \eqref{zm} yields that
\[\omega_\mathbb{A}^2(\mathbb{S})\leq \omega_A^2(P)+\frac{1}{4}\|Q\|_A^2+\frac{1}{2}\omega_A(P) \|Q\|_A.\]
This proves \eqref{ff1}. Using an argument similar to that used in proof of Theorem \ref{thmf}, we get the desired result.
\end{proof}

Next we state the following useful lemmas related to $A$-selfadjoint operators.
\begin{lemma}\label{lemma:5}
Let $T, S \in \mathcal{B}(\mathcal{H})$ be two $A$-selfadjoint operators. If $T-S\geq_A 0$, then
\begin{equation*}
\|T\|_A \geq \|S\|_A.
 \end{equation*}
\end{lemma}
\begin{proof}
Since $T-S\geq_A 0$, then $\langle (T-S)x\mid x\rangle_A \geq0$ for all $x\in \mathcal{H}$. This gives
\begin{equation*}
\langle Tx\mid x\rangle_A \geq \langle Sx\mid x\rangle_A,\quad \forall\,x\in \mathcal{H}.
 \end{equation*}
So, by taking the supremum over all $x\in \mathcal{H}$ with $\|x\|_A=1$ in the above inequality and then using \eqref{aself1} we obtain the desired result.
\end{proof}
\begin{lemma}(\cite{feki02})\label{posss}
Let $T \in \mathcal{B}_A(\mathcal{H})$ be an $A$-selfadjoint operator. Then, $T^{2n}\geq_A 0$ for any positive integer $n$.
\end{lemma}
We are now in a position to prove the following theorem.
 \begin{theorem}\label{theorem:upper 3}
Let $\mathbb{T}=\begin{pmatrix}
P &Q\\
R &S
\end{pmatrix}$ be such that $P, Q, R, S\in \mathcal{B}_{A}(\mathcal{H})$. Then,
\[\omega_\mathbb{A}(\mathbb{T})\leq \sqrt{2\omega_A^2(P)+\frac{1}{2} \left( \|P^{\sharp_A}Q\|_A+\|Q\|_A^2  \right)}+\sqrt{2\omega_A^2(S)+\frac{1}{2} \left(\|S^{\sharp_A}R\|_A+\|R\|_A^2 \right)}.\]
\end{theorem}
\begin{proof}

We first prove that
\begin{equation}\label{2020f}
\omega_\mathbb{A}\left[\begin{pmatrix}
P &Q\\
0 &0
\end{pmatrix}\right]\leq \sqrt{2\omega_A^2(P)+\frac{1}{2}\left(\|P^{\sharp_A}Q\|_A+\|Q\|_A^2\right)}.
\end{equation}
Let $\theta \in \mathbb{R}$. By using \eqref{diez2}, it can be verified that
$$\Re_\mathbb{A}\left[e^{i\theta}\begin{pmatrix}
P &Q\\
0 &0
\end{pmatrix}\right]=\begin{pmatrix}
\Re_A(e^{i\theta}P)&\frac{1}{2}e^{i\theta}Q\\
\frac{1}{2}e^{-i\theta}Q^{\sharp_A}&0
\end{pmatrix}$$
and
$$ \Im_\mathbb{A}\left[e^{i\theta}\begin{pmatrix}
P &Q\\
0 &0
\end{pmatrix}\right]=-i\begin{pmatrix}
i\Im_A(e^{i\theta}P)&\frac{1}{2}e^{i\theta}Q\\
-\frac{1}{2}e^{-i\theta}Q^{\sharp_A}&0
\end{pmatrix}.$$
Moreover, by Lemma \ref{posss}, $\Im_\mathbb{A}^2\left[e^{i\theta}\begin{pmatrix}
P &Q\\
0 &0
\end{pmatrix}\right]\geq_\mathbb{A}\begin{pmatrix}
0 &0\\
0 &0
\end{pmatrix}$. So, we have
$$\Re_\mathbb{A}^2\left[e^{i\theta}\begin{pmatrix}
P &Q\\
0 &0
\end{pmatrix}\right]+\Im_\mathbb{A}^2\left[e^{i\theta}\begin{pmatrix}
P &Q\\
0 &0
\end{pmatrix}\right]-\Re_\mathbb{A}^2\left[e^{i\theta}\begin{pmatrix}
P &Q\\
0 &0
\end{pmatrix}\right]\geq_\mathbb{A}\begin{pmatrix}
0 &0\\
0 &0
\end{pmatrix}.$$
Hence, it follows from Lemma \ref{lemma:5} that
$$\left\|\Re_\mathbb{A}^2\left[e^{i\theta}\begin{pmatrix}
P &Q\\
0 &0
\end{pmatrix}\right]\right\|_\mathbb{A}\leq \left\|\Re_\mathbb{A}^2\left[e^{i\theta}\begin{pmatrix}
P &Q\\
0 &0
\end{pmatrix}\right]+\Im_\mathbb{A}^2\left[e^{i\theta}\begin{pmatrix}
P &Q\\
0 &0
\end{pmatrix}\right]\right\|_\mathbb{A}.$$
On the other hand, a short calculation reveals that
\begin{align*}
&\Re_\mathbb{A}^2\left[e^{i\theta}\begin{pmatrix}
P &Q\\
0 &0
\end{pmatrix}\right]+\Im_\mathbb{A}^2\left[e^{i\theta}\begin{pmatrix}
P &Q\\
0 &0
\end{pmatrix}\right]\\
&=\begin{pmatrix}
\Re_A^2(e^{i\theta}P)+\Im_A^2(e^{i\theta}P)&0\\
0&0
\end{pmatrix}+\begin{pmatrix}
0&\frac{P^{\sharp_A}Q}{2}\\
\frac{Q^{\sharp_A}P}{2}&0
\end{pmatrix}+\begin{pmatrix}
\frac{QQ^{\sharp_A}}{2}&0\\
0&\frac{Q^{\sharp_A}Q}{2}
\end{pmatrix}.
\end{align*}
Hence, by using Lemma \ref{max} and \eqref{zm} we see that
\begin{align}\label{333}
&\left\|\Re_\mathbb{A}\left[e^{i\theta}\begin{pmatrix}
P &Q\\
0 &0
\end{pmatrix}\right]\right\|_\mathbb{A}^2\nonumber\\
&\leq  \left\|\Re_A^2(e^{i\theta}P)+\Im_A^2(e^{i\theta}P)\right\|_A+\frac{1}{2}\max\{\|P^{\sharp_A}Q\|_A,\|Q^{\sharp_A}P\|_A\}+\frac{1}{2}\|Q\|_A^2\nonumber\\
   &\leq 2\omega_A^2(P)+\frac{1}{2} \left(\max\{\|P^{\sharp_A}Q\|_A,\|Q^{\sharp_A}P\|_A\}+\|Q\|_A^2 \right).
\end{align}
On the other hand, one observes that $P_{\overline{\mathcal{R}(A)}}A=AP_{\overline{\mathcal{R}(A)}}=A$. Moreover, by \eqref{newsemi}, we see that
\begin{align*}
\|P^{\sharp_A}Q\|_A
&=\|Q^{\sharp_A}P_{\overline{\mathcal{R}(A)}}PP_{\overline{\mathcal{R}(A)}}\|_A\\
&=\sup\left\{|\langle AP_{\overline{\mathcal{R}(A)}}x\mid (Q^{\sharp_A}P_{\overline{\mathcal{R}(A)}}P)^{\sharp_A}y\rangle|\,;\;x,y\in \mathcal{H},\,\|x\|_{A}=\|y\|_{A}= 1\right\}\\
&=\sup\left\{|\langle Q^{\sharp_A}P_{\overline{\mathcal{R}(A)}}Px\mid y\rangle_A|\,;\;x,y\in \mathcal{H},\,\|x\|_{A}=\|y\|_{A}= 1\right\}\\
&=\sup\left\{|\langle AP_{\overline{\mathcal{R}(A)}}Px\mid Qy\rangle|\,;\;x,y\in \mathcal{H},\,\|x\|_{A}=\|y\|_{A}= 1\right\}\\
&=\sup\left\{|\langle Q^{\sharp_A}Px\mid y\rangle_A|\,;\;x,y\in \mathcal{H},\,\|x\|_{A}=\|y\|_{A}= 1\right\}\\
&=\|Q^{\sharp_A}P\|_A.
\end{align*}
So, by taking into account \eqref{333}, it follows that
\begin{align*}
\left\|\Re_\mathbb{A}\left[e^{i\theta}\begin{pmatrix}
P &Q\\
0 &0
\end{pmatrix}\right]\right\|_\mathbb{A}^2
   &\leq 2\omega_A^2(P)+\frac{1}{2}\left( \|P^{\sharp_A}Q\|_A+\|Q\|_A^2 \right).
\end{align*}
By taking the supremum over all $\theta \in \mathbb{R}$ in the above inequality we obtain \eqref{2020f} as required. Finally, by using an argument similar to that used in proof of Theorem \ref{thmf}, we get the desired inequality.
\end{proof}

Our next result reads as follows.

\begin{theorem}\label{theorem:upper 3}
Let $\mathbb{T}=\begin{pmatrix}
P &Q\\
R &S
\end{pmatrix}$ be such that $P, Q, R, S\in \mathcal{B}_{A}(\mathcal{H})$. Then,
\[\omega_\mathbb{A}(\mathbb{T})\leq \min\{\mu,\nu \},\]
where
\begin{align*}
\mu
& =\sqrt{\min\{\|P+Q\|_A^2 ,\|P-Q\|_A^2\}+2\omega_A(PQ^{\sharp_A}) } \\
 &\quad\quad\quad+\sqrt{\min\{\|R+S\|_A^2 ,\|R-S\|_A^2\}+2\omega_A(SR^{\sharp_A}) },
\end{align*}
and
\begin{align*}
\nu
& =\sqrt{\min\{\|P+R\|_A^2 ,\|P-R\|_A^2\}+2\omega_A(P^{\sharp_A}R)} \\
 &\quad\quad\quad+\sqrt{\min\{\|Q+S\|_A^2 ,\|Q-S\|_A^2\}+2\omega_A(S^{\sharp_A}Q)}.
\end{align*}
\end{theorem}
\begin{proof}
By using \eqref{refine1} together with \eqref{diez} and Lemma \ref{max} we see that
\begin{align}\label{kd}
 \omega_\mathbb{A}\left[\begin{pmatrix}
    P& Q\\
    0& 0
    \end{pmatrix}\right]
    &\leq \left\|\begin{pmatrix}
    P& Q\\
    0 & 0
    \end{pmatrix}\right\|_\mathbb{A}\nonumber\\
    &=\left\|\begin{pmatrix}
    P& Q\\
    0& 0
    \end{pmatrix}
    \begin{pmatrix}
    P& Q\\
    0& 0
    \end{pmatrix}^{\sharp_\mathbb{A}}\right\|_\mathbb{A}^{\frac{1}{2}}\nonumber\\
    &=\left\|\begin{pmatrix}
    P& Q\\
    0& 0
    \end{pmatrix}\begin{pmatrix}
    P^{\sharp_A} & 0\\
    Q^{\sharp_A}& 0
    \end{pmatrix}\right\|_\mathbb{A}^{\frac{1}{2}}\nonumber\\
    &=\left\|\begin{pmatrix}
    PP^{\sharp_A}+QQ^{\sharp_A} & 0\\
    0& 0
    \end{pmatrix}\right\|_\mathbb{A}^{\frac{1}{2}}\nonumber\\
    &=\|PP^{\sharp_A}+QQ^{\sharp_A}\|_A^{\frac{1}{2}}.
\end{align}
Moreover, it is not difficult to verify that
$$
PP^{\sharp_A}+QQ^{\sharp_A}=(P\pm Q)(P\pm Q)^{\sharp_A}\mp (PQ^{\sharp_A}+QP^{\sharp_A}).
$$
So, since $PP^{\sharp_A}+QQ^{\sharp_A}\geq_A$, it follows from \eqref{aself1} that
\begin{align*}
\|PP^{\sharp_A}+QQ^{\sharp_A}\|_A
& =\omega_A(PP^{\sharp_A}+QQ^{\sharp_A}) \\
 &=\omega_A\Big( (P\pm Q)(P\pm Q)^{\sharp_A}\mp (PQ^{\sharp_A}+QP^{\sharp_A})\Big)\\
 &\leq \omega_A\Big( (P\pm Q)(P\pm Q)^{\sharp_A}\Big)+\omega_A(PQ^{\sharp_A})+\omega_A(QP^{\sharp_A})\\
  & =\left\|P\pm Q\right\|_A^2 +\omega_A(PQ^{\sharp_A})+\omega_A(QP^{\sharp_A}),
\end{align*}
where the last equality follows by using \eqref{aself1} together with \eqref{diez} since the operator $(P\pm Q)(P\pm Q)^{\sharp_A}$ is $A$-positive. Further, one observes that
\begin{align*}
\omega_A(PQ^{\sharp_A})
& =\omega_A\left((Q^{\sharp_A})^{\sharp_A}P^{\sharp_A}\right) \\
 &=\omega_A(P_{\overline{\mathcal{R}(A)}}QP_{\overline{\mathcal{R}(A)}}P^{\sharp_A})=\omega_A(P_{\overline{\mathcal{R}(A)}}QP^{\sharp_A}).
\end{align*}
This yields that $\omega_A(PQ^{\sharp_A})=\omega_A(QP^{\sharp_A})$. Thus, we get
\begin{align*}
\|PP^{\sharp_A}+QQ^{\sharp_A}\|_A
 &\leq \left\|P\pm Q\right\|_A^2+2\omega_A(PQ^{\sharp_A}).
\end{align*}
This implies that
\begin{equation*}
\|PP^{\sharp_A}+QQ^{\sharp_A}\|_A\leq \min\Big(\left\|P+ Q\right\|_A^2, \left\|P- Q\right\|_A^2\Big)+2\omega_A(PQ^{\sharp_A}).
\end{equation*}
So, by taking into account \eqref{kd}, we get
\begin{equation}\label{a5iran}
\omega_\mathbb{A}\left[\begin{pmatrix}
    P& Q\\
    0& 0
    \end{pmatrix}\right]\leq \sqrt{\min\Big(\left\|P+ Q\right\|_A^2, \left\|P- Q\right\|_A^2\Big)+2\omega_A(PQ^{\sharp_A})}.
\end{equation}
By considering the $\mathbb{A}$-unitary operator $\mathbb{U}=\begin{pmatrix}
0&I \\
I&0
\end{pmatrix}$, we see that
\begin{align*}
\omega_\mathbb{A}\left[\begin{pmatrix}
    P & Q\\
    R & S
    \end{pmatrix}\right]
    &\leq \omega_\mathbb{A}\left[\begin{pmatrix}
    P & Q\\
    0 & 0
    \end{pmatrix}\right]+\omega_\mathbb{A}\left[\begin{pmatrix}
    0 & 0\\
    R & S
    \end{pmatrix}\right]\\
    &= \omega_\mathbb{A}\left[\begin{pmatrix}
    P & Q\\
    0 & 0
    \end{pmatrix}\right]+\omega_\mathbb{A}\left[\mathbb{U}^{\sharp_\mathbb{A}}\begin{pmatrix}
    S & R\\
    0 & 0
    \end{pmatrix}\mathbb{U}\right] \\
    &=\omega_\mathbb{A}\left[\begin{pmatrix}
    P & Q\\
    0 & 0
    \end{pmatrix}\right]+\omega_\mathbb{A}\left[\begin{pmatrix}
    S & R\\
    0 & 0
    \end{pmatrix}\right]\quad (\text{by }\,\eqref{weak})\\
    &\leq \min\Big(\left\|P+ Q\right\|_A^2, \left\|P- Q\right\|_A^2\Big)+2\omega_A(PQ^{\sharp_A})\\
    &\quad\quad+ \min\Big(\left\|R+ S\right\|_A^2, \left\|R- S\right\|_A^2\Big)+2\omega_A(SR^{\sharp_A})
\end{align*}
By observing that $\omega_\mathbb{A}\left[\begin{pmatrix}
    P & Q\\
    R & S
    \end{pmatrix}\right]=\omega_\mathbb{A}\left[\begin{pmatrix}
    P^{\sharp_A} & R^{\sharp_A}\\
    Q^{\sharp_A} & S^{\sharp_A}
    \end{pmatrix}\right]$ and using similar arguments as above we get
\begin{align*}
\omega_\mathbb{A}\left[\begin{pmatrix}
    P & Q\\
    R & S
    \end{pmatrix}\right]
& = \omega_\mathbb{A}\left[\begin{pmatrix}
    P^{\sharp_A} & R^{\sharp_A}\\
    Q^{\sharp_A} & S^{\sharp_A}
    \end{pmatrix}\right]\\
 &\leq \min\Big(\left\|P^{\sharp_A}+ R^{\sharp_A}\right\|_A^2, \left\|P^{\sharp_A}-R^{\sharp_A}\right\|_A^2\Big)+2\omega_A(P^{\sharp_A}(R^{\sharp_A})^{\sharp_A})\\
    &\quad\quad+ \min\Big(\left\|Q^{\sharp_A}+ S^{\sharp_A}\right\|_A^2, \left\|Q^{\sharp_A}- S^{\sharp_A}\right\|_A^2\Big)+2\omega_A(S^{\sharp_A}(Q^{\sharp_A})^{\sharp_A})\\
 &=\min\Big(\left\|P+ R\right\|_A^2, \left\|P-R\right\|_A^2\Big)+2\omega_A(R^{\sharp_A}P)\\
    &\quad\quad+ \min\Big(\left\|Q+ S\right\|_A^2, \left\|Q- S\right\|_A^2\Big)+2\omega_A(Q^{\sharp_A}S).
\end{align*}
Hence, the proof is complete since $\omega_A(R^{\sharp_A}P)=\omega_A(P^{\sharp_A}R)$ and $\omega_A(Q^{\sharp_A}S)=\omega_A(S^{\sharp_A}Q)$.
\end{proof}
In order to prove a lower bound for $\omega_\mathbb{A}\left[\begin{pmatrix}
    P& Q\\
    0& 0
    \end{pmatrix}\right]$, we need the following lemmas.

\begin{lemma}\label{456}
Let $T,S\in\mathcal{B}_{A}(\mathcal{H})$. Then
\begin{equation}
\max\Big\{\left\Vert T+S\right\Vert_A^{2}, \left\Vert T-S\right\Vert_A^{2}\Big\}-\|TT^{\sharp_A}+SS^{\sharp_A}\|_A\leq 2\omega_A\left(TS^{\sharp_A}\right).  \label{2.7}
\end{equation}
\end{lemma}

\begin{proof}
Let $x\in \mathcal{H}$ be such that $\|x\|_A=1$. We obviously have%
\begin{align*}
\left\Vert Tx+Sx\right\Vert_A ^{2}& =\left\Vert Tx\right\Vert_A ^{2}+2\Re\left(%
\left\langle Tx\mid Sx\right\rangle_A\right) +\left\Vert Sx\right\Vert_A ^{2} \\
& \leq \left\langle \left( T^{\sharp_A}T+S^{\sharp_A}S\right) x\mid x\right\rangle_A
+2\left\vert \left\langle \left( S^{\sharp_A}T\right) x\mid x\right\rangle_A\right\vert\\
& \leq \omega_A \left( T^{\sharp_A}T+S^{\sharp_A}S\right)+2\omega_A \left( S^{\sharp_A}T\right)\\
&=\left\|T^{\sharp_A}T+S^{\sharp_A}S\right\|_A+2\omega_A \left( S^{\sharp_A}T\right),
\end{align*}%
where the last equality follows since $T^{\sharp_A}T+S^{\sharp_A}S\geq_A0$. So, by taking the supremum over all $x\in \mathcal{H}$ with $\left\Vert x\right\Vert_A =1$ in the above inequality we get%
\begin{align*}
\left\Vert T+S\right\Vert_A ^{2}
&\leq \omega_A \left( T^{\sharp_A}T+S^{\sharp_A}S\right)+2\omega_A \left( S^{\sharp_A}T\right).
\end{align*}%
Similarly, we prove that
\begin{align*}
\left\Vert T-S\right\Vert_A ^{2}
&\leq\left\|T^{\sharp_A}T+S^{\sharp_A}S\right\|_A+2\omega_A \left( S^{\sharp_A}T\right).
\end{align*}%
Hence, we get the desired inequality \eqref{2.7}.
\end{proof}

\begin{lemma}
Let $T,S\in \mathcal{B}(\mathcal{H})$. Then, the following assertions hold
\begin{itemize}
  \item [(1)] If $T\geq_A0$ and $S\geq_A0$, then
  \begin{equation}\label{eq109}
\|T-S\|_A\leq \max\{\|T\|_A,\|S\|_A\}.
\end{equation}
  \item [(2)] If $T,S\in \mathcal{B}_A(\mathcal{H})$, then
  \begin{equation}\label{fin}
2\|T^{\sharp_A}S\|_A\leq \|TT^{\sharp_A}+SS^{\sharp_A}\|_A.
\end{equation}
\end{itemize}
\end{lemma}
\begin{proof}
\noindent (1)\;Let $Q=T-S$. It is not difficult to see that
$$\|T\|_AI\geq_AT\geq_AQ\;\,\text{ and }\;\,\|S\|_AI\geq_AS\geq_A-Q.$$
This implies, by Lemma \ref{lemma:5}, that $\|Q\|_A\leq \|T\|_A$ and $\|Q\|_A\leq \|S\|_A$. This proves the desired property.
\par \vskip 0.1 cm \noindent (2)\;Let $\mathbb{T}=\begin{pmatrix}T&S\\0&0\end{pmatrix}$. In view of \eqref{diez2} we see that
$$\mathbb{T}\mathbb{T}^{\sharp_\mathbb{A}}=\begin{pmatrix}TT^{\sharp_A}+SS^{\sharp_A}&0\\0&0\end{pmatrix}\text{ and } \mathbb{T}^{\sharp_\mathbb{A}}\mathbb{T}=\begin{pmatrix}T^{\sharp_A}T&T^{\sharp_A}S\\S^{\sharp_A}T&S^{\sharp_A}S\end{pmatrix}.$$
Let $\mathbb{U}=\begin{pmatrix}
I&O \\
O&-I
\end{pmatrix}.$ By using \eqref{diez2}, one gets $\mathbb{U}^{\sharp_{\mathbb{A}}}=\begin{pmatrix}
P_{\overline{\mathcal{R}(A)}}&O \\
O&-P_{\overline{\mathcal{R}(A)}}
\end{pmatrix}.$ So, we verify that $\|\mathbb{U}x\|_\mathbb{A}=\|\mathbb{U}^{\sharp_\mathbb{A}}x\|_\mathbb{A}=\|x\|_\mathbb{A}$ for all $x=(x_1,x_2)\in \mathcal{H}\oplus \mathcal{H}$. Hence, $\mathbb{U}$ is $\mathbb{A}$-unitary operator. Moreover, clearly we have $(\mathbb{U}^{\sharp_{\mathbb{A}}})^{\sharp_{\mathbb{A}}}=\mathbb{U}^{\sharp_{\mathbb{A}}}$. In addition, a short calculation shows that
$$
(\mathbb{T}^{\sharp_\mathbb{A}}\mathbb{T})^{\sharp_\mathbb{A}}-\mathbb{U}^{\sharp_{\mathbb{A}}}(\mathbb{T}^{\sharp_\mathbb{A}}\mathbb{T})^{\sharp_\mathbb{A}} \mathbb{U}^{\sharp_{\mathbb{A}}}=\begin{pmatrix}
0&2(T^{\sharp_A}S)^{\sharp_A}\\
2(S^{\sharp_A}T)^{\sharp_A}&0
\end{pmatrix}.
$$
So, by applying Lemma \ref{max} and then using \eqref{eq109} we get
\begin{align*}
2\|T^{\sharp_A}S\|_A
& =\left\|(\mathbb{T}^{\sharp_\mathbb{A}}\mathbb{T})^{\sharp_\mathbb{A}}-\mathbb{U}^{\sharp_{\mathbb{A}}}(\mathbb{T}^{\sharp_\mathbb{A}}\mathbb{T})^{\sharp_\mathbb{A}} \mathbb{U}^{\sharp_{\mathbb{A}}}\right\|_\mathbb{A} \\
 &\leq \max\left\{\left\|(\mathbb{T}^{\sharp_\mathbb{A}}\mathbb{T})^{\sharp_\mathbb{A}}\right\|_\mathbb{A}, \left\|\mathbb{U}^{\sharp_{\mathbb{A}}}(\mathbb{T}^{\sharp_\mathbb{A}}\mathbb{T})^{\sharp_\mathbb{A}} \mathbb{U}^{\sharp_{\mathbb{A}}}\right\|_\mathbb{A}\right\}\\
  &=\max\left\{\left\|\mathbb{T}^{\sharp_\mathbb{A}}\mathbb{T}\right\|_\mathbb{A}, \left\|\mathbb{U}(\mathbb{T}^{\sharp_\mathbb{A}}\mathbb{T})\mathbb{U}\right\|_\mathbb{A}\right\}\\
  &\leq \left\|\mathbb{T}^{\sharp_\mathbb{A}}\mathbb{T}\right\|_\mathbb{A}\quad (\text{since }\,\left\|\mathbb{U}\right\|_\mathbb{A}=1)\\
   &=\left\|\mathbb{T}\mathbb{T}^{\sharp_\mathbb{A}}\right\|_\mathbb{A}=\|TT^{\sharp_A}+SS^{\sharp_A}\|_A\quad(\text{by Lemma } \ref{max}).
\end{align*}
Hence, we prove the desired result.
\end{proof}

\begin{lemma}\label{ffii}
Let $T,S\in \mathcal{B}_{A}(\mathcal{H})$. Then,
\begin{align*}
&\max\left\{\|T+S\|_A^2 , \|T-S\|_A^2\right\}\\
&\geq \frac{\left|\,\|T+S\|_A^2-\|T-S\|_A^2 \right|}{2}+\max\left\{\|T^2+S^2\|_A , \|T^{\sharp_A}T+S^{\sharp_A}S\|_A, \|TT^{\sharp_A}+SS^{\sharp_A}\|_A\right\}.
\end{align*}
\end{lemma}
\begin{proof}
Notice that for any two real numbers $x$ and $y$ we have
\begin{equation}\label{r}
\max\{x,y\}=\frac{1}{2}\left(x+y+|x-y|\right).
\end{equation}
Now, by using \eqref{diez} together with \eqref{r} we see that
\begin{align}\label{repl}
&\max\left\{\|T+S\|_A^2 , \|T-S\|_A^2\right\}\nonumber\\
& = \frac{1}{2}\left(\|T+S\|_A^2+\|T-S\|_A^2+|\,\|T+S\|_A^2-\|T-S\|_A^2\,|\right)\nonumber\\
 &= \frac{1}{2}\left(\left\|(T^{\sharp_A}+S^{\sharp_A})(T+S)\right\|_A+\|(T^{\sharp_A}-S^{\sharp_A})(T-S)\|_A+|\,\|T+S\|_A^2-\|T-S\|_A^2|\right)\nonumber\\
  &\geq\frac{1}{2}\left(\left\|(T^{\sharp_A}+S^{\sharp_A})(T+S)+(T^{\sharp_A}-S^{\sharp_A})(T-S)\right\|_A+|\,\|T+S\|_A^2-\|T-S\|_A^2\,|\right)\nonumber\\
    &=\left\|T^{\sharp_A}T+S^{\sharp_A}S\right\|_A+\frac{\Big|\,\|T+S\|_A^2-\|T-S\|_A^2\Big|}{2}.
\end{align}
By replacing $T$ and $S$ by $T^{\sharp_A}$ and $S^{\sharp_A}$, respectively, in \eqref{repl} and then using the fact that $\|X\|_A=\|X^{\sharp_A}\|_A$ for every $X\in \mathcal{B}_{A}(\mathcal{H})$ we get
\begin{align*}
\max\left\{\|T+S\|_A^2 , \|T-S\|_A^2\right\}
    &\geq \left\|TT^{\sharp_A}+SS^{\sharp_A}\right\|_A+\frac{\Big|\,\|T+S\|_A^2-\|T-S\|_A^2\Big|}{2}.
\end{align*}
On the other hand, by \eqref{r} one has
\begin{align*}
&\max\left\{\|T+S\|_A^2 , \|T-S\|_A^2\right\}\\
& = \frac{1}{2}\left(\|T+S\|_A^2+\|T-S\|_A^2+|\,\|T+S\|_A^2-\|T-S\|_A^2\,|\right)\\
 &\geq \frac{1}{2}\left(\left\|(T+S)^2\right\|_A+\left\|(T-S)^2\right\|_A+|\,\|T+S\|_A^2-\|T-S\|_A^2\,|\right)\\
  &\geq \frac{1}{2}\left(\left\|(T+S)^2+(T-S)^2\right\|_A+|\,\|T+S\|_A^2-\|T-S\|_A^2\,|\right)\\
  &= \left\|T^2+S^2\right\|_A+\frac{\Big|\,\|T+S\|_A^2-\|T-S\|_A^2\Big|}{2}.
\end{align*}
So, the proof of the lemma is complete.
\end{proof}

Now we are ready to prove the following theorem.
\begin{theorem}\label{lastthm}
Let $P, Q\in \mathcal{B}_{A}(\mathcal{H})$. Then,
\begin{equation}
\omega_\mathbb{A}\left[\begin{pmatrix}
    P& Q\\
    0& 0
    \end{pmatrix}\right]\geq \frac{1}{2}\sqrt{\max\Big(\left\|P+ Q\right\|_A^2, \left\|P- Q\right\|_A^2\Big)-2\omega_A(PQ^{\sharp_A})}.
\end{equation}
\end{theorem}
\begin{proof}
We first prove that
\begin{equation}\label{poss}
\max\Big(\left\|P+ Q\right\|_A^2, \left\|P- Q\right\|_A^2\Big)-2\omega_A(PQ^{\sharp_A})\geq0.
\end{equation}
By applying \eqref{fin} together with the second inequality in \eqref{refine1}, one observes
$$
2\omega_A(PQ^{\sharp_A})\leq \left\|P^{\sharp_A}P+Q^{\sharp_A}Q\right\|_A.
$$
This implies, by applying Lemma \ref{ffii}, that
\begin{align*}
\max\Big(\left\|P+ Q\right\|_A^2, \left\|P- Q\right\|_A^2\Big)
&\geq \left\|P^{\sharp_A}P+Q^{\sharp_A}Q\right\|_A+\frac{\left|\,\|P+Q\|_A^2-\|P-Q\|_A^2 \right|}{2}\\
 &\geq 2\omega_A(PQ^{\sharp_A})+\frac{\left|\,\|P+Q\|_A^2-\|P-Q\|_A^2 \right|}{2}.
\end{align*}
Hence, \eqref{poss} holds. Now, by using the first inequality in \eqref{refine1} we get
\begin{align*}
\omega_\mathbb{A}\left[\begin{pmatrix}
    P& Q\\
    0& 0
    \end{pmatrix}\right]
&\geq \frac{1}{4}\left\|\begin{pmatrix}
    P& Q\\
    0& 0
    \end{pmatrix} \right\|_\mathbb{A}^2\\
 &=\frac{1}{4}\left\|\begin{pmatrix}
    P& Q\\
    0& 0
    \end{pmatrix} \begin{pmatrix}
    P^{\sharp_A}& 0\\
    Q^{\sharp_A}& 0
    \end{pmatrix} \right\|_\mathbb{A}\quad (\text{by }\,\eqref{diez})\\
 &=\frac{1}{4}\left\|\begin{pmatrix}
    PP^{\sharp_A}+ QQ^{\sharp_A}& 0\\
    0& 0
    \end{pmatrix} \right\|_\mathbb{A}\\
 &=\frac{1}{4}\|PP^{\sharp_A}+ QQ^{\sharp_A}\|_A\quad(\text{by Lemma } \ref{max})\\
 &\geq\frac{1}{4} \max\Big\{\left\Vert T+S\right\Vert_A^{2}, \left\Vert T-S\right\Vert_A^{2}\Big\}- 2\omega_A\left(TS^{\sharp_A}\right),
\end{align*}
where the last inequality follows from Lemma \ref{456}. This finishes the proof of the theorem.
\end{proof}
The following corollary is an immediate consequence of Theorem \ref{lastthm} and \eqref{a5iran}.
\begin{corollary}
Let $P, Q\in \mathcal{B}_{A}(\mathcal{H})$ be such that $APQ^{\sharp_A}=0$. Then,
\begin{equation*}
\frac{1}{2}\max\Big(\left\|P+ Q\right\|_A, \left\|P- Q\right\|_A\Big)
\leq\omega_\mathbb{A}\left[\begin{pmatrix}
    P& Q\\
    0& 0
    \end{pmatrix}\right]\leq \min\Big(\left\|P+ Q\right\|_A, \left\|P- Q\right\|_A\Big).
\end{equation*}
In particular, if $Q=0$ we get
\begin{equation*}
\tfrac{1}{2} \|P\|_A\leq\omega_A(P) \leq \|P\|_A.
\end{equation*}
\end{corollary}


\end{document}
